\documentclass[11pt,reqno]{amsart}
\usepackage{graphicx}

\usepackage{amscd,amssymb,amsmath,amsthm}
\usepackage{graphicx}
\usepackage{color}
\usepackage{cite}
\newtheorem{thm}{Theorem}

\newtheorem{lemma}{Lemma}

\numberwithin{equation}{section} 

\def\R{\mathbb{R}}

\def\Z{\mathbb{Z}}

\begin{document}

\title[The dynamical system of floor function]{The dynamical system generated by the floor function $\lfloor\lambda x\rfloor$}

\author{Rozikov U.A., Sattarov I.A., Usmonov J.B.}

\address{U.\ A.\ Rozikov \\ Institute of mathematics,
29, Do'rmon Yo'li str., 100125, Tashkent, Uzbekistan.}
\email {rozikovu@yandex.ru}

 \address{I.A. Sattarov and J.B. Usmonov \\ Namangan state university, Namangan, Uzbekistan.} \email
{iskandar1207@rambler.ru\ \ javohir0107@mail.ru}

\begin{abstract}
We investigate the dynamical system generated by the function
$\lfloor\lambda x\rfloor$ defined on $\R$ and with a parameter $\lambda\in \R$.
For each given $m\in \mathbb N$ we show that there exists a region of values of $\lambda$,
 where the function has exactly $m$ fixed points (which are non-negative integers),
 also there is another region for $\lambda$, where there are exactly $m+1$ fixed points (which are non-positive integers).
 Moreover the full set $\mathbb Z$ of integer numbers is the set of fixed points iff $\lambda=1$.
 We show that depending on $\lambda$ and on the initial point $x$ the limit of the forward orbit of the dynamical
 system may be one of the following possibilities: (i) a fixed point, (ii) a two-periodic orbit
  or (iii) $\pm\infty$.
\end{abstract}

\maketitle

\section{Introduction and Preliminaries}
Let $X\subset \R$ and $f$ be a map from $X$ to itself.
The set $X$ need not be closed or bounded interval, although this is usually assumed
in the literature. The point of view of dynamical systems is to study iterations of $f$:
if $f^n$ denotes the $n$-fold composition (iteration) of $f$ with itself, then for a given point $x$ one
investigates the sequence $x, f(x), f^2(x), f^3(x)$, and so on. This sequence is called
one-dimensional discrete time dynamical system or the forward orbit of $x$, or just the orbit of $x$ for short.

The theory of one-dimensional non-linear dynamical systems
underwent considerable progress, as the result of the
efforts of theorists from several fields -in particular from physics-
to get a better understanding, by making use of the notion of the
"Hopf's bifurcation" of the appearance of cycles and of the transition
to aperiodic or "chaotic" behavior in physical, biological or
ecological systems. These new developments seem to be potentially very
useful for the study of periodic and aperiodic phenomena in economics.
Parts of this theory have been indeed already used in economic or game
theory  \cite{B}-\cite{J}.

In the theory of the dynamical system the main problem is to know the set of limit points of $\{f^n(x)\}_{n\geq 1}$ for each initial point $x$,
 It is particularly interesting when the orbit repeats. In this case $x$ is a periodic point.
 If $f$ has a periodic point of period $m$, then $m$ is called a period for $f$. Given
a continuous map of an interval one may ask what periods it can have, this question was answered in Sharkovsky's well-known theorem.
  One of the implications of the theorem is that if a
discrete dynamical system on the real line has a periodic point of period 3,
then it must have periodic points of every other period. 
But this theorem works only for continuous functions.
The dynamical systems generated by discontinuous functions are rather difficult to study, and each such system
requires a specific method.

In this paper we shall study a one-parametric family of discontinuous functions, which is
defined by the floor function as
$f_\lambda(x)=\lfloor \lambda x\rfloor$, $\lambda \in \R$.

 For convenience of the reader let us give necessary definitions and properties of the floor function.
  The floor function of $x\in \mathbb R$ is defined by
$$ \lfloor x \rfloor=\max\, \{m\in\mathbb{Z}\,:\, m\le x\}.$$
The following are properties of the floor function which we shall use in this paper:
$$\lfloor x \rfloor = m \ \ \mbox{ if and only if } \ \ m \le x < m+1,$$
$$\lfloor x \rfloor = m \ \ \mbox{ if and only if } \ \ x-1 < m \le x,$$
$$x<m \ \ \mbox{ if and only if } \ \ \lfloor x \rfloor <m,$$
$$m\le x \ \ \mbox{ if and only if } m \le \lfloor x \rfloor.$$
The above are not necessarily true if $m$ is not an integer.

The floor function has been applied in the study of mod operator, quadratic reciprocity, rounding, number of digits,
Riemann function etc. Moreover this function is useful to give formulas for prime numbers, here are some of them (see\cite{R}):
there is a number $\theta = 1.3064...$ (Mills' constant) and a number $\omega = 1.92878...$  with the property that
$$\left\lfloor \theta^3 \right\rfloor, \left\lfloor \theta^9 \right\rfloor, \left\lfloor \theta^{27} \right\rfloor, \dots$$
$$\left\lfloor 2^\omega\right\rfloor, \left\lfloor 2^{2^\omega} \right\rfloor, \left\lfloor 2^{2^{2^\omega}} \right\rfloor, \dots$$
are all prime.

These various applications of the floor function gave a motivation to study dynamical systems generated by
such functions. The parameter $\lambda$ makes rich the behavior of our dynamical system generated by $\lfloor \lambda x\rfloor$:
in subsection 2.1 for each given $m\in \mathbb N$ we show that there exists a region of values of $\lambda$,
 where the function has exactly $m$ fixed points (which are non-negative integers),
 also there is another region for $\lambda$, where there are exactly $m+1$ fixed points (which are non-positive integers).
 Moreover the full set $\mathbb Z$ of integer numbers is the set of fixed points iff $\lambda=1$.
 In the rest subsections of the Section 2 we show that depending on $\lambda$ and on the initial point $x$ the limit of the forward orbit of the dynamical
 system may be a fixed point or a two-periodic orbit or $\pm\infty$.
 
\section{The dynamical system}

In this paper we consider the dynamical system associated with the function $f:\R\to\R$ defined by
\begin{equation}\label{3.1}
f(x)\equiv f_\lambda(x)=\lfloor\lambda x\rfloor,
\end{equation}
where $\lambda\in \R$ is a parameter.

\subsection{Fixed points} A point $x\in \R$ is called a fixed point of $f$ if $f(x)=x$.
The set of all fixed points is denoted by Fix$(f)$.
The following lemma gives all fixed points of this function.
\begin{lemma}\label{l1} For the set of fixed points the following hold
\begin{itemize}
\item[1)] If $\lambda\leq 0$ then Fix$(f)=\{0\}$;

\item[2)] If ${m-1\over m}<\lambda\leq{m\over m+1}$ for some $m\in \mathbb N$ then  Fix$(f)=\{0,-1,-2,...,-m\}$;

\item[3)] If $\lambda=1$ then Fix$(f)=\Z$;

\item[4)] If ${m+1\over m}\leq\lambda<{m\over m-1}$ for some $m\in \mathbb N$ then  Fix$(f)=\{0,1,2,...,m-1\}$.
\end{itemize}
\end{lemma}
\begin{proof}
1) Let $\lambda\leq 0$. In the case $\lambda=0$ we have $f(x)=0$, i.e., only $x=0$ is fixed point.
Moreover, $x=0$ is a fixed point independently on value of $\lambda$.
For $\lambda<0$ we consider the following cases:
\begin{itemize}
\item[a)] If $x<0$ then $f(x)=\lfloor\lambda x\rfloor\geq 0$, consequently $f(x)\neq x$;

\item[b)] If $x>0$ then $f(x)=\lfloor\lambda x\rfloor<0$, consequently $f(x)\neq x$.
\end{itemize}
Thus if $\lambda\leq 0$ then the equation $\lfloor\lambda x\rfloor=x$ has a unique solution $x=0$.

2) Let $0<\lambda<1$. Since solutions of $\lfloor\lambda x\rfloor=x$ are integer numbers we consider the following partition of the set
$\Z=\Z^-\cup\{0\}\cup \Z^+$. If $x\in \Z^+$ then by $\lambda\in(0,1)$ we have $0<\lambda x<x$ and  $\lfloor\lambda x\rfloor<x$.
     For each $\lambda\in(0,1)$ there exists $m\in\mathbb N$ such that ${m-1\over m}<\lambda\leq{m\over m+1}$ holds, because
\begin{equation}\label{3.2}
(0,1)=\bigcup_{m=1}^{\infty}\left({m-1\over m}, \,\right.\left. {m\over m+1}\right].
\end{equation}
Assume ${m-1\over m}<\lambda\leq{m\over m+1}$ then $\forall x\in \Z^-$ we have
\begin{equation}\label{3.3}
x-{x\over m+1}\leq\lambda x<x-{x\over m}.
\end{equation}
From (\ref{3.3}) for any $x\in\{-1,-2,...,-m\}$ we obtain
$$x<x-{x\over m+1}\leq\lambda x<x-{x\over m}\leq x+1.$$
Thus each $x\in\{-1,-2,...,-m\}$ satisfies $\lfloor\lambda x\rfloor=x$.

Let now $x<-m$ then there exists $l\in \mathbb N$ such that $x=-(m+l)$. By (\ref{3.3}) we get
$$x+{m+l\over m+1}\leq\lambda x<x+{m+l\over m}$$ and
 $$x+1\leq\lambda x<x+1+{l\over m}.$$
 Thus $$\lfloor\lambda x\rfloor\geq x+1>x.$$

This completes the proof of part 2).

3) Straightforward.

4). Let $\lambda>1$.
For $x\in \Z^-$ by $\lambda>1$ we get $\lambda x<x$, hence $\lfloor\lambda x\rfloor<x$.
 Since
\begin{equation}\label{3.4}
(1,+\infty)=\bigcup_{m=1}^{\infty}\left[{m+1\over m}, \right.\left.\, {m\over m-1}\right)
\end{equation}
there is  $m\in \mathbb N$ such that ${m+1\over m}\leq\lambda<{m\over m-1}$ and for $x\in \Z^+$ we have
\begin{equation}\label{3.5}
x+{x\over m}\leq\lambda x<x+{x\over m-1}.
\end{equation}
Consequently for $x\in\{1,2,...,m-1\}$ we have $$x<x+{x\over m}\leq\lambda x<x+{x\over m-1}\leq x+1.$$
 Thus each $x\in\{1,2,...,m-1\}$ is a fixed point.

Now assume $x>m-1$, then there exists  $p\in \mathbb N$ such that $x=m-1+p$. Thus by (\ref{3.5})
we obtain
$$x+{m+p-1\over m}\leq\lambda x<x+{m-1+p\over m-1}.$$
 Since ${p-1\over m}\geq0$ we obtain  $$x+1\leq\lambda x<x+1+{p\over m-1},$$ consequently
 $$\lfloor\lambda x\rfloor\geq x+1>x.$$
Thus if ${m+1\over m}\leq\lambda<{m\over m-1}$ for some $m\in \mathbb N$ then each $x\in \{0,1,2,...,m-1\}$ is a fixed point, these are all possible fixed points.
\end{proof}

\subsection{The limit points}
For a given function $f:\mathbb R\to\mathbb R$ the $\omega$-limit set of $x\in \mathbb R$, denoted by $\omega(x,f)$ or $\omega(x)$,
is the set of cluster points of the forward orbit $\{f^n(x)\}_{n\in \mathbb{N}}$ of the iterated function $f$. Hence, $y\in \omega(x)$ if and only if there is a strictly increasing sequence of natural numbers $\{n_k\}_{k\in \mathbb{N}}$ such that $f^{n_k}(x)\rightarrow y$ as $k\rightarrow\infty$.

In this section for function (\ref{3.1}) we shall describe the set $\omega(x)$ for each given $x\in \mathbb R$.

\subsubsection{The case $\lambda\leq 0$.} The case $\lambda=0$ is trivial $\omega(0)=\{0\}$.
Consider the case $\lambda<0$.

\begin{thm}\label{t1a} If $\lambda<0$ then
 the dynamical system generated by $f$ has the following properties:
\begin{itemize}
\item[1.] If $-1<\lambda<0$ then $\forall x\in \R$ we have
$$\lim_{n\to\infty}f^n(x)=0,$$
i.e., $\omega(x)=\{0\}$.
\item[2.] If $\lambda=-1$ then each non-zero integer has period two, i.e. $f^2(x)=x$ for any $x\in \mathbb Z\setminus \{0\}$. Moreover
$f^3(x)=f(x),$ for each $x\in \R,$ i.e.,
$$\omega(x)=\left\{\begin{array}{ll}
\{x,f(x)\},\ \ \mbox{if} \ \ x\in\mathbb Z\\[3mm]
\{f(x),f^2(x)\}, \ \ \mbox{if} \ \ x\in \R\setminus \mathbb Z.
\end{array}
\right.
$$
\item[3.] If $\lambda<-1$ then $\forall x\in({1\over \lambda},0]$ we have $f(x)=0,$ and
   $$\omega(x)=\left\{\begin{array}{ll}
\{0\},\ \ \mbox{if} \ \ x\in({1\over \lambda},0]\\[3mm]
\{-\infty, +\infty\}, \ \ \mbox{if} \ \ x\in \R\setminus({1\over \lambda},0].
\end{array}
\right.
$$

\end{itemize}
\end{thm}
\begin{proof} 1. Since for each $x\in \R$ the sequence $\{f^n(x)\}_{n\geq 1}$ is subset of $\mathbb Z$, by the condition $-1<\lambda<0$ and properties a) and b) mentioned in the proof of Lemma \ref{l1}, for $x<0$ we have
\begin{equation}\label{3.7}
f(x)>|f^2(x)|\geq f^3(x)>|f^4(x)|\geq f^5(x)>...
\end{equation}
and for $x>0$ we have
 \begin{equation}\label{3.8}
|f(x)|\geq f^2(x) >|f^3(x)|\geq f^4(x)>|f^5(x)|\geq...
\end{equation}
From  (\ref{3.7}) and  (\ref{3.8}) it follows that $\forall x\in \R$  the sequence $\{|f^n(x)|\}_{n\geq 1}$ of non-negative integer numbers is non-increasing and bounded from below by $0$. Hence there is limit $\lim_{n\to \infty}|f^n(x)|=\alpha(x)\in \{0\}\cup\mathbb N$. We shall show that $\alpha(x)=0$ for any $x\in \R$. Suppose there is $x_0\in \R$ such that $\alpha(x_0)\geq 1$. Then the sequence $\{f^n(x_0)\}_{n\geq 1}$ (without absolute value) has the set of limit points $\{-\alpha(x_0), \alpha(x_0)\}$. Moreover by the properties a) and b) mentioned in the proof of Lemma \ref{l1} if $x_0>0$ (the case $x_0<0$ is similar) then
\begin{equation}\label{-}
\lim_{k\to\infty}f^{2k}(x_0)=\alpha(x_0)\ \ \mbox{and} \ \ \lim_{k\to\infty}f^{2k+1}(x_0)=-\alpha(x_0).
\end{equation}
Since $\lfloor \lambda x\rfloor\leq\lambda x$ for any $x\in \R$, we have  $\lfloor \lambda f^n(x)\rfloor\leq\lambda f^n(x)$, i.e.
$f^{n+1}(x)\leq \lambda f^n(x)$. Since $\lambda<0$, using again the properties a) and b), for $x>0$ we have $f(x)<0$ and
\begin{equation}\label{eq}\begin{array}{lllll}
f^2(x)\leq \lambda f(x)=|\lambda||f(x)|,\\[2mm]
f^4(x)\leq |\lambda||f^3(x)|<|\lambda|f^2(x),\\[2mm]
f^6(x)<|\lambda|f^4(x)<|\lambda|^2f^2(x),\\[2mm]
\dots\ \ \dots\\[2mm]
f^{2k}(x)\leq |\lambda|^{k-1}f^2(x)\leq |\lambda|^k|f(x)|.
\end{array}
\end{equation}
Write the last inequality for $x=x_0$, i.e
$$f^{2k}(x_0)\leq |\lambda|^k |f(x_0)|.$$
Since $\lambda\in (-1,0)$, taking limit from both side of this inequality as $k\to \infty$ and using (\ref{-})
we get $\lim_{k\to\infty}f^{2k}(x_0)=\alpha(x_0)\leq 0$, this contradicts to our assumption $\alpha(x_0)\geq 1$. Thus $\alpha(x)=0$ for any $x\in \R$.
Consequently $\lim_{n\to\infty}f^n(x)=0$, since $\lim_{n\to\infty}|f^n(x)|=0$.

2. For $\lambda=-1$ we have $f(x)=\lfloor -x\rfloor$. Then $\forall x\in \Z$ we get $f^2(x)=[\lfloor-\lfloor-x\rfloor\rfloor=x$.
Moreover it is easy to check that $\lfloor-x\rfloor=\lfloor-\lfloor-\lfloor-x\rfloor\rfloor\rfloor$ for all $x\in \R$. Hence $f^3(x)=f(x)$.

3. Assume $\lambda<-1$ and ${1\over\lambda}<x\leq0$. In this case we have $1>\lambda x\geq0$. Consequently $\lfloor\lambda x\rfloor=0$.

Let now $\lambda<-1$ and $x\in \R\setminus({1\over\lambda},0]$. Then for $x\leq{1\over \lambda}$ we get
\begin{equation}\label{3.9}
f(x)<|f^2(x)|\leq f^3(x)<|f^4(x)|\leq f^5(x)<...
\end{equation}
and for  $x>0$ we have
 \begin{equation}\label{3.10}
|f(x)|\leq f^2(x)<|f^3(x)|\leq f^4(x)<|f^5(x)|\leq...
\end{equation}
Since $\{|f^n(x)|\}\subset \mathbb N$, from (\ref{3.9}) and (\ref{3.10}) it follows that $\lim_{n\to\infty}f^n(x)=\infty$.
Moreover, using properties a) and b) one can see that
$$\lim_{k\to\infty}f^{2k}(x)=\left\{\begin{array}{ll}
+\infty, \ \ \mbox{if} \ \ x>0\\
-\infty, \ \ \mbox{if} \ \ x\leq{1\over \lambda}
\end{array}\right.$$
$$\lim_{k\to\infty}f^{2k+1}(x)=\left\{\begin{array}{ll}
-\infty, \ \ \mbox{if} \ \ x>0\\
+\infty, \ \ \mbox{if} \ \ x\leq {1\over \lambda}.
\end{array}\right.$$
\end{proof}

\subsubsection{The case $0<\lambda<1$.}
Note that for each $\lambda\in (0,1)$ there exists $m\in\mathbb N$ such that ${m-1\over m}<\lambda\leq{m\over m+1}$.

\begin{thm}\label{t2} Let ${m-1\over m}<\lambda\leq{m\over m+1}$ for some $m\in \mathbb N$. Then
$$\lim_{n\to\infty}f^n(x)=\left\{\begin{array}{lll}
0, \ \ \mbox{for all} \ \ x\in[0,+\infty),\\[2mm]
k, \ \ \mbox{for all} \ \ x\in[{k\over\lambda},{k+1\over\lambda}),\\[2mm]
-m, \ \ \mbox{for all} \ \ x\in(-\infty,{-m\over\lambda})
\end{array}\right.
$$
\end{thm}
where $k\in\{-1,-2,\dots,-m\}.$
\begin{proof} For any $x\in [0,+\infty)$ we have $f(x)=\lfloor \lambda x\rfloor\leq \lambda x$, iterating this inequality
we get $0\leq f^n(x)\leq\lambda^n x.$
Consequently
$$0\leq\lim_{n\to\infty}f^n(x)\leq\lim_{n\to\infty}\lambda^n x=0,$$
i.e. $\lim_{n\to\infty}f^n(x)=0.$

Consider now $x$ and $k\in \{-1,-2,\dots,-m\}$ such that ${k\over\lambda}\leq x<{k+1\over\lambda}$. Then by $0<\lambda<1$ we
get $k\leq\lambda x<k+1$. Consequently,  $\lfloor\lambda x\rfloor=k$. Since each $k\in \{-1,-2,\dots,-m\}$ is a fixed point, we obtain  $$\lim_{n\to\infty}f^n(x)=k.$$

Consider now the case $x<{-m\over \lambda}$. Then $f(x)=\lfloor\lambda x\rfloor\in \mathbb Z$ with $f(x)<-m$.
Moreover for  $\lambda\in(0,1)$ we have $f(x)>x$ (see the proof of part 2 of Lemma \ref{l1}).
Iterating the last inequality we obtain $f^{n+1}(x)>f^{n}(x)$, i.e. $f^n(x)$ is an increasing
sequence, which is bounded from above by $-m$. Since $-m$ is the unique fixed point in $(-\infty,-m]$, we have
$$\lim_{n\to\infty}f^n(x)=-m.$$
\end{proof}

\subsubsection{The case $\lambda\geq1$.}

For $\lambda=1$ we have $f(x)=\lfloor x\rfloor$ and Fix$(f)=\mathbb Z$. It is easy to see that
 $$\lim_{n\to\infty}f^n(x)=\lfloor x\rfloor, \ \ \forall x\in \R.$$

Let now  $\lambda>1$. Because of (\ref{3.4}) it is sufficient to study the dynamics of $f$ at $\lambda$ such that
${m+1\over m}\leq\lambda<{m\over m-1}$, for some $m\in \mathbb N$. Here for $m=1$ we consider $2\leq\lambda<+\infty$.
\begin{thm}\label{t3} If ${m+1\over m}\leq\lambda<{m\over m-1}$ for some $m\in \mathbb N$ then
$$\lim_{n\to\infty}f^n(x)=\left\{\begin{array}{ll}
k, \ \ \mbox{for all} \ \ x\in[{k\over\lambda},{k+1\over\lambda}),\\[2mm]
-\infty, \ \ \mbox{for all} \ \ x\in(-\infty,0),\\[2mm]
+\infty, \ \ \mbox{for all} \ \ x\in [{m\over\lambda},+\infty),
\end{array}\right.
$$
where $k\in\{0,1,\dots,m-1\}.$
\end{thm}
\begin{proof} Take $x\in [{k\over\lambda},{k+1\over\lambda})$ for some $k\in \{1,\dots,m-1\}$. Then $\lfloor\lambda x\rfloor=k$ and since each $k\in \{1,\dots,m-1\}$ is a fixed point, we obtain $$\lim_{n\to\infty}f^n(x)=k.$$

In the case $x<0$ we have $\lambda x<0$ and
$$x>\lambda x\geq  \lfloor\lambda x\rfloor=f(x).$$
Since $f(x)$ is a non-decreasing function we get from the last
inequality that $f^n(x)>f^{n+1}(x)$, i.e., the sequence $f^n(x)$ is decreasing.
By Lemma \ref{l1} for $\lambda>1$ we know that there is no fixed point of $f$ in $(-\infty,0)$.
Consequently,  $$\lim_{n\to\infty}f^n(x)=-\infty.$$

Assume now ${m\over\lambda}\leq x<+\infty$. Since each $f^n(x)$, $n\geq 1$ is an integer number, for any integer $x\geq m$ we have
$\lfloor\lambda x\rfloor>x$ (see the part 4) of proof of Lemma \ref{l1}). From this inequality it follows that $f^n(x)$ is an increasing sequence and
by Lemma \ref{l1} there is no fixed point in $[m,+\infty)$, hence
 $$\lim_{n\to\infty}f^n(x)=+\infty.$$
\end{proof}

\section*{ Acknowledgements}
U.Rozikov thanks Aix-Marseille University Institute for Advanced Study IM\'eRA
(Marseille, France) for support by a residency scheme. His work also partially supported by the Grant No.0251/GF3 of Education and Science Ministry of Republic
of Kazakhstan.


\begin{thebibliography}{999}

\bibitem{B} J.Benhabib, R.H. Day, \textit{Rational choice and erratic
behaviour}, Review of Economic Studies, {\bf 48}, (1981) 459-472.

\bibitem{Da} R.H. Day, \textit{The emergence of chaos from classical economic
growth}, Quarterly Journal of Economics, {\bf 98}, (1983) 201-213.

\bibitem{D} R.L. Devaney, \textit{An introduction to chaotic dynamical system},
Westview Press, 2003.

\bibitem{J} R.U. Jensen,  R. Urban, \textit{Chaotic price behaviour in a
nonlinear cobweb model},  Yale University. 1982.

 \bibitem{K} A.B. Katok, B. Hasselblatt, \textit{Introduction to the Modern Theory of Dynamical Systems}, Cambridge Univ.
Press, Cambridge, 1995.

\bibitem{R} P. Ribenboim, \textit{The New Book of Prime Number Records},  New York: Springer, 1996.

\end{thebibliography}
\end{document}